\documentclass[12pt]{amsart}
\usepackage{a4}
\usepackage{amssymb}
\usepackage{amsmath}
\usepackage{amsthm}
\usepackage{amstext}
\usepackage{amscd}
\usepackage{latexsym}
\usepackage{graphics}
\theoremstyle{plain}
\newtheorem{thm}{Theorem}[section]
\newtheorem{lemma}[thm]{Lemma}
\newtheorem{lem}[thm]{Lemma}
\newtheorem{cor}[thm]{Corollary}
\newtheorem{prop}[thm]{Proposition}
\theoremstyle{definition}
\newtheorem{rmk}[thm]{Remark}

\newcommand{\sO}{{\mathcal O}}

\newcommand{\N}{{\mathbb N}}

\newcommand{\Z}{{\mathbb Z}}
\newcommand{\Gal}{{\rm Gal}}
\newcommand{\ilim}{\mathop{\varprojlim}\limits} 

\input{xy}
\xyoption{all}

\begin{document}
\title[On Hesselholt's conjecture]{Remark on equicharacteristic analogue of
Hesselholt's conjecture on cohomology of Witt vectors}   
\author[A. Hogadi]{Amit Hogadi}
\address{School of Mathematics, Tata Institute of Fundamental
Research, Homi Bhabha Road, Colaba, Mumbai 400005, India}
\email{amit@math.tifr.res.in}
\author[S. Pisolkar]{Supriya Pisolkar}
\address{School of Mathematics, Tata Institute of Fundamental
Research, Homi Bhabha Road, Colaba, Mumbai 400005, India}
\email{supriya@math.tifr.res.in}
\date{}
\begin{abstract}
Let $L/K$ be a finite Galois extension of complete discrete valued fields of
characteristic $p$. Assume that the induced residue field extension $k_L/k_K$ is
separable. For an integer $n\geq 0$, let $W_n(\sO_L)$ denote the ring of Witt vectors of length $n$ with coefficients in $\sO_L$. We show that the proabelian group $\{H^1(G,W_n(\sO_L))\}_{n\in \N}$ is zero. This is an equicharacteristic analogue of Hesselholt's conjecture (see
\cite{lh}) which was proved in \cite{hp} when the discrete valued fields are of
mixed characteristic. 
\end{abstract}

\maketitle

\section{Introduction} \label{intro}

Let $K$ be a complete discrete valued field with residue field of characteristic
$p > 0$. $L/K$ be a finite Galois extension with Galois group $G$. Suppose that
$k_L/k_K$ is separable. When $K$ is of characteristic zero,  Hesselholt
conjectured in \cite{lh} that the  proabelian group
$\{H^1(G,W_n(\sO_L))\}_{n\in \N}$ vanishes, where $W_n(\sO_L)$ is the ring of
Witt vectors of length $n$ with coefficients in $\sO_L$ (w.r.t. to the prime $p$).
 As explained in
\cite{lh}, this can be viewed as an analogue of Hilbert theorem $90$ for the
Witt ring $W(\sO_L)$. This conjecture was proved in some cases in \cite{lh} and
in general in \cite{hp}. In this paper we show that a similar vanishing holds
when $K$ is of characteristic $p$. 
The main result of this paper is as follows.

\begin{thm} \label{main} Let $L/K$ be a finite Galois extension of complete
discrete valued equi-characteristic fields with Galois group $G$. Assume that
the induced residue field extension is separable. Then the pro-abelian group $
\{H^1(G,W_n(\sO_L))\}$ is zero.
\end{thm}
In order to prove this conjecture one easily reduces to the case where $L/K$ is
a totally ramified Galois extension of degree $p$ (see \cite[Lemma 3.1]{hp}). We make
the argument in \cite{hp} work in the equicharacteristic case using an explicit
description of the Galois cohomology of $\sO_L$ when $L/K$ is an Artin-Schreir
extension (see Proposition \ref{h1}). 

We remind the reader that a proabelian group indexed by $\N$ is an inverse
system of abelian groups $\{A_n\}_{n\in \N}$ whose vanishing means that for
every $n\in \N$, there exists an integer $m>n$ such that the map $$ A_m\to A_n$$
is zero (see \cite[Section 1]{jannsen}). This is clearly implies the vanishing
of $\ilim_n H^1(G,W_n(\sO_L))$. It also implies the vanishing of
$H^1(G,W(\sO_L))$ (with $W(\sO_L)$ being considered as discrete $G$-mdoule) by
 \cite[Corollary 1.2]{hp}. 

\begin{rmk}
One may also consider an analogue of Theorem \ref{main} when $K$ is of equi-characteristic zero. However, in this case, all extensions $L/K$ are tamely ramified and the vanishing 
$$H^1(G(L/K),W_n(\sO_L))=0 \ \forall \ n\geq 0.$$ 
can be easily deduced from the fact that $\sO_L$ is a projective $\sO_K[G]$ module (see \cite[I. Theorem(3)]{frohlich}).
\end{rmk}

\noindent {\bf Acknowledgement}: We thank the referee for several useful comments and suggestions.

\section{Cohomology of integers in Artin-Schreier extensions}

Let $K$ be a complete discrete valued field of characteristic $p$ as before. Let
$\sO_K$ and $k$ denote the discrete valuation ring and residue field of $K$
respectively. Let $L/K$ be a Galois extension of degree $p$. Recall that the
ramification break (or lower ramification jump) of this extension, to be
denoted by $s=s(L/K)$, is  the smallest non-negative integer such that the
induced action of $\Gal(L/K)$ on $\sO_L/m_L^{s+1}$ is faithful, where $m_L$ is
the maximal ideal of $\sO_L$ (\cite{fesenko}, II, 4.5). Thus unramified extensions are precisely the
extensions with ramification break equal to zero. We recall the following well
known result.

\begin{prop}{\rm (}see \cite{hasse} or \cite[Proposition 2.1]{lt}{\rm )}
\label{v_K(f)}Let $L/K$ be a Galois extension of degree $p$ of complete discrete
valued fields of characteristic $p$. There exists an element $f\in K$ such that
$L$ is obtained by joining a root of the polynomial 
 $$ X^p-X-f=0$$ 
Further one can choose $f$ such that $v_K(f)$ is coprime to $p$. In this case
$$v_K(f)=-s$$ where $s$ is the ramification break of $\Gal(L/K)$. 
\end{prop}

We now fix an $f\in K$ given by the above proposition. Clearly, if $v_K(f)>0$
then by Hensel's lemma $X^p-X-f$ already has a root in $K$. If $v_K(f)=0$ then
the extension given by adjoining the root of this polynomial is an unramified
extension. 

\begin{prop}\label{ol}
Let $L/K$, $f\in K$ be as above. Assume $L/K$ is totally ramified. Let $\lambda$
be a root of $X^p-X-f$ in $L$. Let $s$ be the ramification break of $\Gal(L/K)$.
Then the discrete valuation ring $\sO_L$, is the subset of $L$ is given by 
$$ \sO_L = \{ \displaystyle{\sum_{i=0}^{p-1}}a_i\lambda^i  \ | \ a_i\in \sO_K \
\text{with} \ v_K(a_i) \geq \frac{is}{p} \}.$$
\end{prop}
\begin{proof}
Clearly the set $\{1,\lambda,\cdots,\lambda^{p-1}\}$ is a $K$-basis of $L$. Thus
any element $x\in L$ can be  written  uniquely in the form 
$$ x = \sum_{i=0}^{p-1}a_i\lambda^i. $$
Note that $v_L(\lambda)=v_K(f)=-s$ is coprime to $p$ by the choice of $f$. Since
$L/K$ is ramified, $s$ is nonzero. Moreover $v_L(a_i)=pv_K(a_i)$
is divisible by $p$. We thus conclude that for each $0\leq i \leq p-1$, the
values of $v_L(a_i \lambda^i)$ are all distinct modulo $p$, and hence, distinct.
 
Thus $$v_L(\sum_{i=0}^{p-1}a_i\lambda^i)\geq 0 \  \text{ if and only if }
v_L(a_i\lambda^i)\geq 0 \ \ \text{for all } \ 0 \leq i<p$$
But $v_L(a_i\lambda^i)=pv_K(a_i)-is$. This proves the claim.
\end{proof}

\begin{lem}\label{sum}
Let $p$ be a prime number as before.  Let 
$$S_k = \displaystyle{\sum}_{n=0}^{p-1}n^k$$ Then
\begin{enumerate}
 \item $S_k\equiv 0 \ {\rm mod} \ p \ \ {\text if}\  0\leq k \leq p-2$
 \item $S_{p-1} \equiv -1 \ {\rm mod}\ p $
\end{enumerate}
\end{lem}
\begin{proof}
The first congruence follows from the recursive formula (see
\cite[(4)]{barbara})
$$ S_k=\frac{1}{k+1}\left( p^{k+1}-p^k -
\displaystyle{\sum}_{j=0}^{k-2}\binom{k}{j}S_{j+1} \right)$$
and using the fact that when $k\leq p-2$, $k+1$ is invertible modulo $p$. $(2)$
follows from Fermat's little theorem.
\end{proof}

We now state an explicit description of $H^1(G,\sO_L)$.

\begin{prop}\label{h1}
Notation as in \eqref{v_K(f)}. Let $\sigma$ be a generator of $\Gal(L/K)$. Let
$\sO_L^{\operatorname{tr}=0}$ denote the set of all trace zero elements in
$\sO_L$ and $$(\sigma-1)\sO_L= \{\sigma(x)-x \ | \ x\in \sO_L\}.$$  Then,
\begin{enumerate}
 \item $\sO_L^{\operatorname{tr}=0} =
\{\displaystyle{\sum_{i=0}^{p-2}}a_i\lambda^i \ | \ v_K(a_i)\geq \frac{is}{p}
\}$
 \item $(\sigma-1)\sO_L=   \{ \displaystyle{\sum_{i=0}^{p-2}} a_i\lambda^i \ | \
v_K(a_i)\geq \frac{(i+1)s}{p} \} $
\end{enumerate}
\end{prop}
\begin{proof} Since the sets  $\sO_L^{\operatorname{tr}=0}$ and  $(\sigma-1)\sO_L$ are independent of the choice of $\sigma$, we may assume, without loss of generality, that  $\sigma$ is the generator satisfying $\sigma(\lambda)=\lambda+1$. \\ 
$(1)$ Let $x=\sum_{i=1}^{p-1}a_i\lambda^i$. Let $S_k$ be as in
Lemma \ref{sum}.  Then  \\
$ 
\begin{array}{rcl}
\operatorname{tr}(x) & =  & \displaystyle{\sum}_{j=0}^{p-1}\sigma^j(x)  \\
      & =  & \displaystyle{\sum}_{j=0}^{p-1}\displaystyle{\sum}_{i=0}^{p-1}a_i(\lambda+j)^i     \\
      & =  & \displaystyle{\sum}_{i=0}^{p-1}a_i\left(
\sum_{j=0}^{p-1}(\lambda+j)^i  \right)  
\end{array} \\
$
By binomially expanding and collecting coefficients of $\lambda^i$, we get \\
$
\begin{array}{rcll}
  \operatorname{tr}(x)    & =  & \displaystyle{\sum}_{i=0}^{p-1}a_i\left( p\lambda^i+
\displaystyle{\sum}_{j=1}^{i}\binom{i}{j}S_j\lambda^{i-j} \right)  \\
  
      & =  & -a_{p-1} \ \ \  \ \ \ ... \text{(by Lemma \ref{sum})}.
\end{array} 
$\\
This together with Proposition \ref{ol} proves $(1)$. \\

\noindent $(2)$ Suppose $x=\displaystyle{\sum}_{i=1}^{p-1}a_i\lambda^i\in
(\sigma-1)\sO_L$. Then 
$$ \displaystyle{\sum}_{i=1}^{p-1}a_i\lambda^i =
(\sigma-1)\displaystyle{\sum}_{i=1}^{p-1}b_i\lambda^i,$$
where $v_K(b_i)\geq \frac{is}{p}$ by \eqref{ol}. This gives us the following
system of $p$ equations \\ \\
$
\begin{array}{lcl}
a_0 & = & b_1 + \cdots + b_{p-1} \\
a_1 & = & \binom{2}{1}b_2+ \binom{3}{2}b_3 + \cdots + \binom{p-1}{p-2}b_{p-1} \\
. &  &  \\
a_i & = & \binom{i+1}{i}b_{i+1} + \cdots + \binom{p-1}{p-(i+1)}b_{p-1}\\
. &  &  \\
a_{p-2} & = & (p-1)b_{p-1}\\
a_{p-1} & = & 0 \\
\end{array}
$ \\ \\
Since $v_K(b_{i+1})\geq \frac{(i+1)s}{p}$, we get $v_K(a_i)\geq
\frac{(i+1)s}{p}$. 
Thus 
$$(\sigma-1)\sO_L \subset   \{ \displaystyle{\sum_{i=0}^{p-2}} a_i\lambda^i \ |
\ v_K(a_i)\geq \frac{(i+1)s}{p} \}.$$
Conversely assume  
$$ \displaystyle{\sum}_{i=1}^{p-1}a_i\lambda^i \in \{
\displaystyle{\sum_{i=0}^{p-2}} a_i\lambda^i \ | \ v_K(a_i)\geq \frac{(i+1)s}{p}
\}.$$
Since $H^1(G,L)=0$, we know there exists $\sum b_i\lambda^i \in L$ such that 
$$ \displaystyle{\sum}_{i=1}^{p-1}a_i\lambda^i =
(\sigma-1)\displaystyle{\sum}_{i=1}^{p-1}b_i\lambda^i.$$
These $b_i's$ satisfy the above system of $p$ equations. Using that
$v_K(a_i)\geq \frac{(i+1)s}{p}$, it is straightforward to observe by induction
that $v_K(b_i) \geq \frac{is}{p}$.  Hence
$\displaystyle{\sum}_{i=1}^{p-1}b_i\lambda^i \in \sO_L$. 
\end{proof}

The following corollary is the equi-characteristic analogue of \cite[Lemma 2.4]{lh}.
\begin{cor}\label{s}
Let $L/K$ be as in Proposition \ref{v_K(f)}. Let $x\in
\sO_L^{\operatorname{tr}=0}$ be an element that defines a non zero class in
$H^1(G,\sO_L)$. Then $v_L(x) \leq s-1$.
\end{cor}
\begin{proof} We will show that for any $x\in \sO_L^{\operatorname{tr}=0}$, if
$v_L(x)\geq s$, then the class of $x$ in $H^1(G,\sO_L)$ is zero. By \eqref{h1},
we may write $$x=\displaystyle{\sum}_{i=1}^{p-2}a_i\lambda^i \ \ \text{with} \ \
v_L(a_i)\geq is.$$
Since for all $i$, $v_L(a_i\lambda^i)$ are distinct (see proof of \eqref{ol}),
we have
$$ v_L(x) = \text{inf}\{v_L(a_i\lambda^i)\}.$$
Thus $v_L(x)\geq s$ implies 
$$ v_L(a_i\lambda^i)=v_L(a_i)-is \geq s \ \ \forall \  i.$$
This shows that $v_L(a_i) \geq (i+1)s$ which by Proposition \ref{h1} implies
$x \in (\sigma-1)\sO_L$, and hence defines a trivial class in $H^1(G,\sO_L)$.
This proves the claim.
\end{proof}

\section{Proof of the main theorem}

Following \cite[Lemma 3.1]{hp}, we reduce the proof of the main theorem to the case
when $L/K$ is totally ramified Galois extension of degree $p$. Thus throughout
this section we fix an extension $L/K$ which is of this type. We also fix a
generator $\sigma \in \Gal(L/K)$. We first define a polynomial $G \in
\Z[X_1,...,X_p]$ in $p$ variables by 
$$ G(X_1,...,X_p) = \frac{1}{p}\left(
(\displaystyle{\sum}_{i=1}^pX_i)^p-(\displaystyle{\sum}_{i=1}^pX_i^p)\right).$$
Note that despite the occurence of $\frac{1}{p}$, $G$ is a polynomial with
integral coefficients. \\

\noindent Now for an element $x\in L$ define 
$$F(x) = G(x,\sigma(x),...,\sigma^i(x),...,\sigma^{p-1}(x)).$$
The expression $F(x)$ is formally equal to the expression
$\frac{\operatorname{tr}(x)^p-\operatorname{tr}(x^p)}{p}$ and makes sense in
characteristic $p$ since $G$ has integral coefficients. Moreover, since for any
$x\in L$, $F(x)$ is invariant under the action of $\Gal(L/K)$, $F(x)\in K$. We
now observe that \cite[Lemma 2.2]{lh} holds in characteristic $p$ in the
following form

\begin{lemma}\cite[Lemma 2.2]{lh}\label{tr}
For all $x\in \sO_L$, $v_K(F(x))=v_L(x)$.
\end{lemma}

\begin{proof}[Proof of \ref{main}]
The proof follows \cite[Proof of 1.4]{hp} verbatim, with \eqref{s} and
\eqref{tr} replacing \cite[Lemma 3.2]{hp} and \cite[Lemma 3.4]{hp} respectively.
We briefly recall the idea of the proof for the convenience of the reader.
By \cite[Lemma 1.1]{lh}, it is enough to show that for large $n$, the map 
$$ H^1(G,W_n(\sO_L))\to H^1(G,\sO_L)$$ 
is zero. By \eqref{s}, it is enough to show that for large $n$ 
$$ (x_0,...,x_{n-1}) \in W_n(\sO_L)^{\operatorname{tr}=0} \implies v_L(x_0)\geq s. $$
The condition $ (x_0,...,x_{n-1}) \in W_n(\sO_L)^{\operatorname{tr}=0}$ can be
rewritten as 
$$ \sum_{i=0}^{p-1}(\sigma^i(x_0),...,\sigma^i(x_{n-1}))=0.$$
Using the formula for addition of Witt vectors, one analyses the above equation
and obtains (see \cite[Lemma 3.5]{hp})

\begin{equation}\label{one} \operatorname{tr}(x_{\ell})=
F(x_{\ell-1})-C.\operatorname{tr}(x_{\ell-1})^p+h_{\ell-2} \ \ \ , \  1 \leq \ell \leq
n-1
\end{equation}
where $C$ is a fixed integer and $h_{\ell-2}$ is a polynomial in
$x_0,...,x_{\ell-2}$ and its conjugates such that each monomial appearing in
$h_{\ell-2}$ is of degree $\geq p^2$.  
Using the above equation, Lemma \ref{tr} and \cite[Lemma 2.1]{lh} one proves the
theorem in the following three steps, for details of which we refer the reader
to \cite[Proof of 1.4]{hp}. \\

\noindent Step $1$. We claim that for $0\leq \ell \leq n-2$ 
$$v_L(x_{\ell})\geq \frac{s(p-1)}{p}.$$
Recall equation \eqref{one}  for $1 \leq \ell \leq n-1$
$$\operatorname{tr}(x_{\ell})=
F(x_{\ell-1})-C.\operatorname{tr}(x_{\ell-1})^p+h_{\ell-2}.$$ One now proves the
above claim by induction on $\ell$.  Using the fact that $h_{-1} = 
\operatorname{tr}(x_0) = 0$, the above equation gives 
$$-\operatorname{tr}(x_1) = F(x_0). $$
This, together with \cite[Lemma 2.1]{lh} proves the claim for $\ell=0$. The rest
of the induction argument is straightforward.  This claim, together with
equation \eqref{one}, is then used  to observe $v_K(h_{\ell})\geq s(p-1)$ for
all $\ell$.\\

\noindent Step $2$. In this step we show that for $2\leq i \leq n-1$
$$ v_L(x_{n-i}) \geq
\frac{s(p-1)}{p}\left(1+\frac{1}{p}+\cdots+\frac{1}{p^{i-2}}\right).$$
This is proved by induction on $i$, using equation \eqref{one} and the estimates
$$v_L(x_{\ell})\geq \frac{s(p-1)}{p}
\ \ ,  \ \  v_L(h_{\ell})\geq s(p-1)$$ obtained in Step $1$.

\noindent Step $3$. Since values taken by $v_L$ is a discrete valuation, for an
integer $M$ such that 
$$\frac{s(p-1)}{p}\left(1+\frac{1}{p}+\cdots+\frac{1}{p^{M-2}}\right)>s-1, $$
we have $v_L(x_0)\geq s$.
\\

\end{proof}


\begin{thebibliography}{xx}

\bibitem{fesenko} Fesenko, I. B.; Vostokov, S. V.; {\it Local fields and their
extensions}. Second edition. Translations
of Mathematical Monographs, 121. American Mathematical Society.

\bibitem{frohlich} Fr$\rm \ddot{o}$hlich, Albrecht; {\it Galois module structure
of algebraic integers}. {\bf 3}. Springer-Verlag, Berlin, 1983.


\bibitem{hasse} H. Hasse; Theorie der relativ-zyklischen algebraischen
Funktionenk\"{o}rper, insbesondere bei endlichem Konstantenk\"{o}rper. {\it J.
Reine Angew. Math.} {\bf 172} (1934), 37?54. 


\bibitem{lh} Lars Hesselholt;  Galois cohomology of Witt vectors of algebraic
integers.
{\it Math. Proc. Cambridge Philos. Soc.} {\bf 137} (2004), no. {\bf 3},
551--557. 

\bibitem{hp}Hogadi, Amit; Pisolkar, Supriya. On the cohomology of Witt vectors
of p-adic integers and a conjecture of Hesselholt. {\it J. Number Theory} {\bf
131} (2011), no. {\bf 10}, 1797--1807.

\bibitem{jannsen} Jannsen, Uwe; Continuous \'etale cohomology.
{\it Math. Ann.} {\bf 280} (1988), no. {\bf 2}, 207--245


\bibitem{lt} Thomas, Lara; Ramification groups in Artin-Schreier-Witt
extensions.  {\it J. Th\'eor. Nombres Bordeaux}  {\bf 17}  (2005),  no. {\bf 2},
689--720.

\bibitem{barbara} Turner, Barbara; Sums of powers of integers via the binomial
theorem.
{\it Math. Mag.} {\bf 53} (1980), no. {\bf 2}, 92--96.

\end{thebibliography}
\end{document}